\documentclass{svmult}
\usepackage{enumerate}
\usepackage{graphicx}
\usepackage{amsmath}
\usepackage{amssymb}
\usepackage{amsfonts}

\newcommand{\eps}{\varepsilon}

\newcommand{\s}{\qquad}

\newcommand{\be}{\begin{equation}}
\newcommand{\ba}{\begin{array}}
\newcommand{\ee}{\end{equation}}
\newcommand{\ea}{\end{array}}
\newcommand{\oln}{\overline}
\newcommand{\tx}{\text}
\newcommand{\ds}{\displaystyle}
\def\eop{{\ \vrule height 3pt width 3pt depth 0pt}}
\begin{document}\author{
{\sc M. Paramasivam}$^1$\;\;\;\;{\;\;\sc S.~Valarmathi}$^2$\;\;\;\;{and}{\;\;\;\;\sc J.J.H.~Miller}$^3$}
\title*{{\bf A parameter--uniform finite difference method for a singularly
perturbed linear system of second order ordinary differential equations of\\ reaction-diffusion type}}
\institute{$^1${Department of Mathematics, Bishop Heber College, Tiruchirappalli-620 017, Tamil Nadu, India. sivambhcedu@gmail.com.}\\
$^2${Department of Mathematics, Bishop Heber College, Tiruchirappalli-620 017, Tamil Nadu, India. valarmathi07@gmail.com.}\\
$^3${Institute for Numerical Computation and Analysis, Dublin, Ireland. jm@incaireland.org.}}
\titlerunning{Numerical solution of a reaction-difussion system}
\maketitle
\begin{center}
\it{dedicated to G. I. Shishkin on his 70th birthday}
\end{center}
\begin{abstract}
A singularly perturbed linear system of second order ordinary differential equations of reaction-diffusion type with given boundary conditions is considered. The
leading term of each equation is multiplied by a small positive parameter. These parameters are assumed to be distinct. The components of the solution exhibit
overlapping layers. Shishkin piecewise-uniform meshes are introduced, which are used in conjunction with a classical finite difference discretisation, to construct
two numerical methods for solving this problem. It is proved that the numerical approximations obtained with these methods are essentially first, respectively
second, order convergent uniformly with respect to all of the parameters.
\end{abstract}

\section{\large\bf Introduction}

The following two-point boundary value problem is considered for the
singularly perturbed linear system of second order differential
equations
\begin{equation}\label{BVP}
-E\vec{u}''(x)+A(x)\vec{u}(x)=\vec{f}(x),\;\;\; x\in(0,1),\;\;\;\vec{u}(0)\;\tx{and}\;\vec{u}(1)\;\tx{given.}
\end{equation}
Here $\;\vec{u}\;$ is a column $\;n-\tx{vector},\;E\;$ and
$\;A(x)\;$ are $\;n\times n\;$ matrices, $\;E =
\tx{diag}(\vec{\eps}),\;\vec\eps = (\eps_1,\;\cdots,\;\eps_n)\;$
with $\;0\;<\;\eps_i\;\le\;1\;$ for all $\;i=1,\ldots,n$. The
$\eps_i$ are assumed to be distinct and,
for convenience, to have the ordering \[\eps_1\;<\;\cdots\;<\;\eps_n.\] Cases with some of the parameters coincident are not considered here.\\
\noindent The problem can also be written in the operator form
\[\vec L\vec u\; = \;\vec{f},\;\;\;\vec{u}(0)\;\tx{and}\;\vec{u}(1)\;\tx{given}\]
where the operator $\;\vec{L}\;$ is defined by
\[\vec{L}\;=\;-E D^2 + A(x)\;\;\;\tx{and}\;\;\;D^2 = \dfrac{d^2}{dx^2}.\]
For all $x \in [0,1]$ it is assumed that the components $a_{ij}(x)$
of $A(x)$
satisfy the inequalities \\
\begin{eqnarray}\label{a1} a_{ii}(x) > \displaystyle{\sum_{^{j\neq
i}_{j=1}}^{n}}|a_{ij}(x)| \; \; \rm{for}\;\; 1 \le i \le n, \;\;
\rm{and} \;\;a_{ij}(x) \le 0 \;\; \rm{for} \; \; i \neq
j\end{eqnarray} and, for some $\alpha$,
\begin{eqnarray}\label{a2} 0 <\alpha <
\displaystyle{\min_{^{x \in [0,1]}_{1 \leq i \leq n}}}(\sum_{j=1}^n
a_{ij}(x)).
\end{eqnarray} Wherever necessary the required smoothness of the problem data is assumed.
It is also assumed, without loss of generality, that
\begin{eqnarray}\label{a3}
\max_{1 \leq i \leq n} \sqrt{\eps_i} \leq \frac{\sqrt{\alpha}}{4}.
\end{eqnarray}
The norms $\parallel \vec{V} \parallel =\max_{1 \leq k \leq n}|V_k|$
for any n-vector $\vec{V}$, $\parallel y
\parallel =\sup_{0\leq x\leq 1}|y(x)|$ for any
scalar-valued function $y$ and $\parallel \vec{y}
\parallel=\max_{1 \leq k \leq n}\parallel y_{k}
\parallel$ for any vector-valued function $\vec{y}$ are introduced.
Throughout the paper $C$ denotes a generic positive constant, which
is independent of $x$ and of all singular perturbation and
discretization parameters. Furthermore, inequalities between vectors
are understood in the componentwise sense.\\

\noindent For a general introduction to parameter-uniform numerical methods for singular perturbation problems, see \cite{MORS},  \cite{RST} and  \cite{FHMORS}.
Parameter-uniform numerical methods for various special cases of (\ref{BVP}) are examined in, for example, \cite{MMORS}, \cite{MS} and \cite{HV}. For (\ref{BVP})
itself parameter-uniform numerical methods of first and second order are considered in \cite{LM}. However, the present paper differs from \cite{LM} in two important
ways. First of all, the meshes, and hence the numerical methods, used are different from those in \cite{LM}; the transition points between meshes of differing
resolution are defined in a similar but different manner. The piecewise-uniform Shishkin meshes $M_{\vec{b}}$ in the present paper have the elegant property that
they reduce to uniform meshes whenever $\vec{b}=\vec{0}$. Secondly, the proofs of essentially first and second order parameter-uniform convergence do not require
the use of Green's function techniques, as is the case in \cite{LM}. The significance of this is that it is more likely that such techniques can be extended in
future to problems in higher dimensions and to nonlinear problems, than is the case for proofs depending on Green's functions. It is also satisfying, and
appropriate in this special issue, to be able to demonstrate that the methods of proof pioneered by G. I. Shishkin can be extended successfully to problems of this
kind.

The plan of the paper is as follows. In the next section both standard and novel bounds on the smooth and singular components of the exact solution are obtained.
The sharp estimates for the singular component in Lemma \ref{singular} are proved by mathematical induction, while an interesting ordering of the points $x_{i,j}$
is established in Lemma \ref{xs}. In Section 3 appropriate piecewise-uniform Shishkin meshes for essentially first order numerical methods are introduced, the
discrete problem is defined and the discrete maximum principle and discrete stability properties are established. In Section 4 an expression for the local
truncation error is found and two distinct standard estimates are stated. In Section 5 parameter-uniform estimates for the local truncation error of the smooth and
singular components are obtained in a sequence of lemmas. The section culminates with the statement and proof of the essentially first order parameter-uniform error
estimate. In the final section an outline of the construction and error
estimation of an essentially second order parameter-uniform numerical method is presented. \\

\section{\large\bf Analytical results}
The operator $\vec L$ satisfies the following maximum principle
\begin{lemma}\label{max} Let $A(x)$ satisfy (\ref{a1}) and (\ref{a2}). Let
$\;\vec{\psi}\;$ be
any function in the domain of $\;\vec L\;$ such that
$\;\vec{\psi}(0)\ge\vec{0}\;$ and $\;\vec{\psi}(1)\ge\vec{0}.\;$
Then $\;\vec L\vec{\psi}(x)\ge\vec{0}\;$ for
all $\;x\;\in\;(0,1)\;$ implies that $\;\vec{\psi}(x)\ge\vec{0}\;$ for all $\;x\;\in\;[0,1]$.
\end{lemma}
\begin{proof}Let $i^*, x^*$ be such that $\psi_{i^*}(x^{*})=\min_{i,x}\psi_i(x)$
and assume that the lemma is false. Then $\psi_{i^*}(x^{*})<0$ .
From the hypotheses we have $x^* \not\in\;\{0,1\}$ and
$\psi^{\prime \prime}_{i^*}(x^*)\geq 0$.
Thus
\begin{equation*}(\vec{L}\vec \psi(x^*))_{i^*}=
-\eps_{i^*}\psi^{\prime\prime}_{i^*}(x^*)+\sum_{j=1}^n
a_{i^*,j}(x^{*})\psi_j(x^*)<0,
\end{equation*} which contradicts the assumption and proves the
result for $\vec{L}$.\eop \end{proof}

Let $\tilde{A}(x)$ be any principal sub-matrix of $A(x)$ and $\vec{\tilde{L}}$
the corresponding operator. To see that any $\vec{\tilde{L}}$ satisfies the same
maximum principle as $\vec{L}$, it suffices to observe that the elements
of $\tilde{A}(x)$ satisfy \emph{a fortiori} the same inequalities as those of $A(x)$.

We remark that the maximum principle is not necessary for the
results that follow, but it is a convenient tool in their proof.
\begin{lemma}\label{stab} Let $A(x)$ satisfy (\ref{a1}) and (\ref{a2}). If $\vec{\psi}$ is any function in the domain of $\;\vec L,\;$
 then for each $i, \; 1 \leq i \leq n $, \[|\vec{\psi}_i(x)| \le\;
\max\ds\left\{\parallel\vec \psi(0)\parallel,\parallel\vec \psi(1)\parallel, \dfrac{1}{\alpha}\parallel \vec L\vec \psi\parallel\right\},\s x\in [0,1].\]
\end{lemma}
\begin{proof}Define the two functions \[\vec \theta^\pm(x)\;=\;
\max\ds\left\{\parallel\vec \psi(0)\parallel,\;\parallel\vec
\psi(1)\parallel,\;\dfrac{1}{\alpha}\parallel\vec{L}\vec
\psi\parallel\right\}\vec e\;\pm\;\vec \psi(x)\] where $\;\vec
e\;=\;(1,\;\ldots,\;1)^T\;$ is the unit column vector. Using the
properties of $\;A\;$ it is not hard to verify that $\;\vec
\theta^\pm(0)\;\ge\;\vec 0,\;\;\vec \theta^\pm(1)\;\ge\;\vec 0\;$
and $\;\vec L\vec \theta^\pm(x)\;\ge\;\vec 0.\;$ It follows from
Lemma \ref{max} that $\;\vec \theta^\pm(x)\;\ge\;\vec 0\;$ for all
$\;x\;\in\;[0,\;1].\;$\eop
\end{proof}
A standard estimate of the exact solution and its derivatives is
contained in the following lemma.
\begin{lemma}\label{exact} Let $A(x)$ satisfy (\ref{a1}) and (\ref{a2})
and let $\vec u$ be the exact solution of (\ref{BVP}). Then, for each $i=1\; \dots \; n$, all $x\in [0,1]$ \; and \; $k=0,1,2$,
\[|u_i^{(k)}(x)| \leq
C\eps_i^{-\frac{k}{2}}(||\vec{u}(0)||+||\vec{u}(1)||+||\vec{f}||)\]
and
\[|u_i^{\prime\prime\prime}(x)| \leq
C\eps_i^{-\frac{3}{2}}(||\vec{u}(0)||+||\vec{u}(1)||+||\vec{f}||+||\vec{f}^\prime||)\]
\end{lemma}
\begin{proof}
The bound on $\vec{u}$ is an immediate consequence of Lemma
\ref{stab} and the differential equation.\\ To bound
$u_i^\prime(x)$, for all $i$ and any $x$, consider an interval
$N_x=[a,a+\sqrt{\eps_i}]$ such that $x \in N_x$. Then, by the mean
value theorem, for some $y\in N_x$,
\[u_i^\prime(y)=\frac{u_i(a+\sqrt{\eps_i})-u_i(a)}{\sqrt{\eps_i}}\]
and it follows that \[|u_i^\prime(y)|\leq
2\eps_i^{-\frac{1}{2}}||u_i||.
\]
Now
\[\vec{u}^\prime(x)=\vec{u}^\prime(y)+\int_y^x \vec{u}^{\prime\prime}(s)ds=
\vec{u}^\prime(y)+E^{-1}\int_y^x(-\vec{f}(s)+A(s)\vec{u}(s))ds\] and
so
\[|u_i^\prime(x)|\leq |u_i^\prime(y)|+C\eps_i^{-1}(||f_i||+||\vec u||)\int_y^x ds
\leq C\eps_i^{-\frac{1}{2}}(||f_i||+||\vec u||)\] from which the
required bound follows.\\
Rewriting and differentiating the differential equation gives $\vec u^{\prime\prime}= E^{-1}(A\vec u-\vec f), $\;  \; $\vec u^{\prime\prime\prime}= E^{-1}(A\vec
u^\prime+A^\prime\vec u-\vec f^\prime),$ and the bounds on $u_i^{\prime\prime}$, $u_i^{\prime\prime\prime}$ follow.\eop\end{proof}

The reduced solution $\vec{u}_0$ of (\ref{BVP}) is the solution of the reduced equation $A\vec {u}_0=\vec f$. The  Shishkin decomposition of the exact solution
$\;\vec{u}\;$ of (\ref{BVP}) is $\;\vec{u}=\vec{v}+\vec{w}\;$ where the smooth component $\;\vec v\;$ is the solution of $\;\vec L\vec v = \vec f\;$ with $\;\vec
v(0) = \vec u_0(0)\;$ and $\;\vec v(1) = \vec u_0(1)\;$ and  the singular component $\;\vec w\;$ is the solution of $\vec L\vec w\;=\;\vec 0$ with $\vec w(0)=\vec
u(0)-\vec v(0)$ and $\vec w(1)=\vec u(1)-\vec v(1).$ For convenience the left and right boundary layers of $\vec w$ are separated using the further decomposition
$\vec w =\vec{w}^l+\vec{w}^r$ where $\vec{L}\vec{w}^l=\vec{0},\; \vec{w}^l(0)=\vec u(0)-\vec v(0),\; \vec{w}^l(1)=\vec{0}$ and $\vec{L}\vec{w}^r= \vec{0},\;
\vec{w}^r(0)=\vec{0},\; \vec{w}^r(1)=\vec u(1)-\vec v(1).$\\Bounds on the smooth component and its derivatives are contained in
\begin{lemma}\label{smooth} Let $A(x)$ satisfy (\ref{a1}) and (\ref{a2}).
Then the smooth component $\vec v$ and its derivatives satisfy, for all $x\in [0,1]$, \; $i=1,\; \dots \; n$ \;and \; $k=0, \;\dots \;3$,
\[|v_i^{(k)}(x)| \leq C(1+\eps_i^{1-\frac{k}{2}}).\]
\end{lemma}
\begin{proof}
The bound on $\vec{v}$ is an immediate consequence of the defining equations for $\vec{v}$ and Lemma
\ref{stab}.\\
The bounds on $\vec{v}^{\prime}$ and $\vec{v}^{\prime\prime}$ are found as follows. Differentiating twice the equation for $\vec{v}$, it is not hard to see that
$\vec{v}^{\prime\prime}$ satisfies \[ \vec{L}\vec{v}^{\prime\prime}=\vec{g},\;\; \text{where} \;\; \vec{g}=
\vec{f}^{\prime\prime}-A^{\prime\prime}\vec{v}-2A^{\prime}\vec{v}^{\prime.}\] Also the defining equations for $\vec{v}$ yield at $x=0,\;\; x=1$
\[\vec{v}^{\prime\prime}(0)=\vec{0}, \;\; \vec{v}^{\prime\prime}(1)=\vec{0}. \] Applying Lemma \ref{stab} to $\vec{v}^{\prime\prime}$ then gives
\begin{equation}\label{smooth1} ||\vec{v}^{\prime\prime}|| \leq C(1+||\vec{v}^{\prime}||). \end{equation}
Choosing $i^*,\; x^*,$ such that $1 \leq i^* \leq n,\; x^* \in (0,1)$ and
\begin{equation}\label{smooth2} v_{i^*}^{\prime}(x^*)=||\vec{v}^{\prime}|| \end{equation}
and using a Taylor expansion it follows that, for any $y \in [0,1-x^*]$ and some $\eta$, $x^*\;<\;\eta\;<\;x^*+y$,
\begin{equation}\label{smooth3} v_{i^*}(x^*+y) = v_{i^*}(x^*)+y\;\vec{v}'(x^*)+\dfrac{y^2}{2}\;v_{i^*}^{\prime\prime}(\eta).\end{equation}
Rearranging (\ref{smooth3}) yields
\begin{equation}\label{smooth4}v_{i^*}^{\prime}(x^*)=\frac{v_{i^*}(x^*+y)-v_{i^*}(x^*)}{y}-\frac{y}{2}v_{i^*}^{\prime\prime}(\eta)
\end{equation}
and so, from (\ref{smooth2}) and (\ref{smooth4}),
\begin{equation}\label{smooth5}
||\vec{v}^{\prime}|| \leq \frac{2}{y}||\vec{v}||+ \frac{y}{2}||\vec{v}^{\prime\prime}||.
\end{equation}
Using (\ref{smooth5}), (\ref{smooth1}) and the bound on $\vec{v}$ yields
\begin{equation}\label{smooth6}
(1-\frac{Cy}{2})||\vec{v}^{\prime\prime}|| \leq C(1+\frac{2}{y}).
\end{equation}
Choosing $y=\min(\frac{1}{C},1-x^*)$, (\ref{smooth6}) then gives $||\vec{v}^{\prime\prime}|| \leq C$ and (\ref{smooth5}) gives $||\vec{v}^{\prime}|| \leq C$ as
required. The bound on $\vec{v}^{\prime\prime\prime}$ is obtained by a similar argument.\eop\end{proof} The layer functions
$\;B^{l}_{i},\;B^{r}_{i}, 
\; i=1,\; \dots , \; n,\;$,
associated with the solution $\;\vec u$, are defined on $[0,1]$ by
\[B^{l}_{i}(x) = e^{-x\sqrt{\alpha/\eps_i}},\;B^{r}_{i}(x) =
B^{l}_{i}(1-x).\]
The following elementary properties of these layer functions, for
all $1 \leq i < j \leq n$ and $0 \leq x < y \leq
1$, should be noted:\\
(i)\;\;\;$B^{l}_i(x)\; <\; B^{l}_j(x),\;\;B^{l}_i(x)\;
>\; B^{l}_i(y), \;\;0\;<\;B^{l}_i(x)\;\leq\;1$.\\
(ii)\;\;$B^{r}_i(x)\; <\; B^{r}_j(x),\;\;B^{r}_i(x)\; <\;
B^{r}_i(y), \;\;0\;<\;B^{r}_i(x)\;\leq\;1$.

Bounds on the singular components $\vec{w}^l,\; \vec{w}^r$ of
$\vec{u}$ and their derivatives are contained in
\begin{lemma}\label{singular} Let $A(x)$ satisfy (\ref{a1}) and
(\ref{a2}).Then there exists a constant $C,$ such that, for each $x
\in [0,1]$ and $i=1,\; \dots , \; n$,
\[\left|w^l_i(x)\right| \;\le\; C B^l_{n}(x),\;\;
\left|w_i^{l,\prime}(x)\right| \;\le\; C\sum_{q=i}^n
\frac{B^l_{q}(x)}{\sqrt{\eps_q}},\]
\[\left|w_i^{l,\prime\prime}(x)\right| \;\le\; C\sum_{q=i}^n
\frac{B^l_{q}(x)}{\eps_q},\;\; \left|\eps_i
w_i^{l,\prime\prime\prime}(x)\right| \;\le\; C\sum_{q=1}^n
\frac{B^l_{q}(x)}{\sqrt{\eps_q}}.\] Analogous results hold for
$w^r_i$ and its derivatives.
\end{lemma}

\begin{proof}First we obtain the bound on $\vec{w}^l$. We define the two
functions $\vec{\theta}^{\pm}=CB^l_n\vec{e} \pm \vec{w}^l$. Then clearly $\vec{\theta}^{\pm}(0) \geq \vec{0}, \;\; \vec{\theta}^{\pm}(1) \geq \vec{0}$ and
$L\vec{\theta}^{\pm}=CL(B^l_n\vec{e})$. Then, for $i=1,\dots, n$, $(L\vec{\theta}^{\pm})_i =C(\sum_{j=1}^{n}a_{i,j}-\alpha\frac{\eps_i}{\eps_n})B^l_n >0$. By Lemma
\ref{max}, $\vec{\theta}^{\pm}\geq \vec{0}$, which leads to the required bound on $\vec{w}^l$.

Assuming, for the moment, the bounds on the first derivatives $w_i^{l,\prime}$, the system of differential equations satisfied by $\vec{w}^l$ is differentiated to
get
\[-E\vec{w}^{l,\prime\prime\prime}+A\vec{w}^{l,\prime}+A^{\prime}\vec{w}^l =\vec{0}.\] The required
 bounds on the $w_i^{l,\prime\prime\prime}$ follow from those on $w^l_i$ and
$w_i^{l,\prime}$. It remains therefore to establish the bounds on
$w_i^{l,\prime}$ and $w_i^{l,\prime\prime}$, for which the following
mathematical induction argument is used. It is assumed that the
bounds hold for all systems up to order $n-1$. It is then shown that
the bounds hold for order $n$. The induction argument is completed
by observing that the bounds for the scalar case $n=1$ are proved in
\cite{MORS}.

It is now shown that under the induction hypothesis the required bounds hold for $w_i^{l,\prime}$ and $w_i^{l,\prime\prime}$. The bounds when i=n are established
first.The differential equation for $w^l_n$ gives $\eps_n w_n^{l,\prime\prime}=(A\vec{w}^l)_n$ and the required bound on $w_n^{l,\prime\prime}$ follows at once from
that for $\vec{w}^l$. For $w_n^{l,\prime}$ it is seen from the bounds in Lemma \ref{exact}, applied to the system satisfied by $\vec{w}^l$, that
$|w_i^{l,\prime}(x)| \leq C\eps_i^{-\frac{1}{2}}$. In particular, $|w_n^{l,\prime}(0)| \leq C\eps_n^{-\frac{1}{2}}$ and $|w_n^{l,\prime}(1)| \leq
C\eps_n^{-\frac{1}{2}}$. It is also not hard to verify that $\vec{L}\vec{w}^{l,\prime}=-A^{\prime}\vec{w}^l$. Using these results, the inequalities $\eps_i <
\eps_n, \; i<n $, and the properties of $A$, it follows that the two barrier functions $\vec{\theta}^{\pm}=CE^{-\frac{1}{2}}B^l_n \vec{e}\pm \vec{w}^{l,\prime}$
satisfy the inequalities $\vec{\theta}^{\pm}(0) \ge \vec{0},\; \vec{\theta}^{\pm}(1) \ge \vec{0} $, and $\vec{L} \vec{\theta}^{\pm} \ge \vec{0}.$ It follows from
Lemma \ref{max} that $ \vec{\theta}^{\pm} \ge \vec{0}$ and in particular that its $n^{th}$ component satisfies $|w_n^{l,\prime}(x)| \leq
C\eps_n^{-\frac{1}{2}}B^l_n(x)$ as required.

To bound $w_i^{l,\prime}$ and $w_i^{l,\prime\prime}$ for $1 \le i
\le n-1$ introduce $\tilde{\vec{w}}^l=(w^l_1, \dots,w^l_{n-1})$.
Then, taking the first $n-1$ equations satisfied by $\vec{w}^l$, it
follows that
\[-\tilde{E}\tilde{\vec{w}}^{l,\prime\prime}+\tilde{A}\tilde{\vec{w}}^l=\vec{g},\]
where $\tilde{E},\;\tilde{A}$ is the matrix obtained by deleting the
last row and column from $E,\;A$, respectively,  and the components
of $\vec{g}$ are $g_i=-a_{i,n}w^l_n$ for $1 \le i \le n-1$. Using
the bounds already obtained for $w^l_n, w_n^{l,\prime}$, it is seen
that $\vec{g}$ is bounded by $CB^l_n (x)$ and $\vec{g}^{\prime}$ by
$C\frac{B^l_n (x)}{\sqrt{\eps_n}}$. The boundary conditions for
$\tilde{\vec{w}}^l$ are
$\tilde{\vec{w}}^l(0)=\tilde{\vec{u}}(0)-\tilde{\vec{u}}^0(0)$,
$\tilde{\vec{w}}^l(1)=\vec{0}$, where $\vec{u}^0$ is the solution of
the reduced problem $\vec{u}^0=A^{-1}\vec{f}$, and are bounded by
$C(\parallel\vec{u}(0)\parallel+\parallel\vec{f}(0)\parallel)$ and
$C(\parallel\vec{u}(1)\parallel+\parallel\vec{f}(1)\parallel)$. Now
decompose $\tilde{\vec{w}}^l$ into smooth and singular components to
get
\[\tilde{\vec{w}}^l=\vec{q}+\vec{r}, \;\;\ \tilde{\vec{w}}^{l,\prime}=\vec{q}^{\prime}+\vec{r}^{\prime}.\]
Applying Lemma \ref{smooth}  to $\vec{q}$ and using the bounds on
the inhomogeneous term $\vec{g}$ and its derivative
$\vec{g}^{\prime}$, it follows that $|\vec{q}^{\prime}(x)| \leq
C\frac{B^l_n(x)}{\sqrt{\eps_n}}$ and $|\vec{q}^{\prime\prime}(x)|
\leq C\frac{B^l_n(x)}{\eps_n}$. Using mathematical induction, assume
that the result holds for all systems with $n-1$ equations. Then
Lemma \ref{singular} applies to $\vec{r}$ and so, for $i=1, \dots,
n-1$,
\[
|r^{\prime}_{i}(x)| \leq C\sum_{q=i}^{n-1} \frac{B^l_{q}(x)}{\sqrt{\eps_q}},\; |r^{\prime\prime}_{i}(x)| \leq C\sum_{q=i}^{n-1}\frac{B^l_{q}(x)}{\eps_q}.\]
Combining the bounds for the derivatives of $q_i$ and $r_i$, it follows that
\[
|w^{l,\prime}_{i}(x)| \leq C\sum_{q=i}^{n-1}
\frac{B^l_{q}(x)}{\sqrt{\eps_q}},\; |w^{l,\prime\prime}_{i}(x)| \leq
C\sum_{q=i}^{n-1}\frac{B^l_{q}(x)}{\eps_q}.\] Thus, the bounds on
$w_i^{l,\prime}$ and $w_i^{l,\prime\prime}$ hold for a system with
$n$ equations, as required. The proof of the analogous results for
the right boundary layer functions is analogous.\eop\end{proof}
\begin{definition}
For $B_i^l$, $B_j^l$ and each $1 \leq i \neq j \leq n$,  the point $x_{i,j}$ is defined by
\begin{equation}\label{x1}\frac{B^l_i(x_{i,j})}{\sqrt{\varepsilon_i}}=\frac{B^l_j(x_{i,j})}{\sqrt{\varepsilon_j}}. \end{equation}
\end{definition}
It is remarked that
\begin{equation}\label{x2}\frac{B^r_i(1-x_{i,j})}{\sqrt{\varepsilon_i}}=\frac{B^r_j(1-x_{i,j})}{\sqrt{\varepsilon_j}}. \end{equation}

In the next lemma it is shown that the points $x_{i,j}$ exist, are uniquely defined, lie in the domain $[0,\frac{1}{2}]$ and have an interesting ordering.
\begin{lemma}\label{xs} Assume that
(\ref{a3}) holds. If, in addition, $\sqrt{\eps_i} \leq\sqrt{\eps_j}/2,$ then, for all $i,j$ with $1 \leq i < j \leq n$, the points $x_{i,j}$ exist, are uniquely
defined, lie in $(0,\frac{1}{2}]$ and satisfy the following inequalities
\begin{equation}\label{x3}
\frac{B^l_{i}(x)}{\sqrt{\eps_i}} > \frac{B^l_{j}(x)}{\sqrt{\eps_j}},\;\; x \in [0,x_{i,j}),\;\; \frac{B^l_{i}(x)}{\sqrt{\eps_i}} < \frac{B^l_{j}(x)}{\sqrt{\eps_j}},
\; x \in (x_{i,j}, 1].\end{equation} In addition the following ordering holds
\begin{equation}\label{x4}x_{i,j}< x_{i+1,j}, \; \mathrm{if} \;\; i+1<j \;\;
\mathrm{and} \;\; x_{i,j}<
x_{i,j+1}, \;\; \mathrm{if} \;\; i<j. \end{equation} \\
Analogous results hold for the $B^r_i$, $B^r_j$ and the points $1-x_{i,j}.$
\end{lemma}
\begin{proof} Existence, uniqueness and (\ref{x3}) follow
from the observation that $\sqrt{\eps_i}<\sqrt{\eps_j}$, for $i<j$, and the ratio of the two sides of (\ref{x1}), namely
\[\frac{B^l_{i}(x)}{\sqrt{\eps_i}}\frac{\sqrt{\eps_j}}{B^l_{j}(x)}=
\frac{\sqrt{\eps_j}}{\sqrt{\eps_i}}\exp{(-\sqrt{\alpha} x(\frac{1}{\sqrt{\eps_i}}-\frac{1}{\sqrt{\eps_j}}))},\] is monotonically decreasing from the value
$\frac{\sqrt{\eps_j}}{\sqrt{\eps_i}} >1$ as $x$ increases
from $0$.\\
The point $x_{i,j}$ is the unique point $x$ at which this ratio has the value $1.$ Rearranging (\ref{x1}) gives
\begin{equation}\label{x5}x_{i,j}=
\frac{\ln(\frac{1}{\sqrt{\eps_i}})-\ln(\frac{1}{\sqrt{\eps_j}})}{\sqrt{\alpha}(\frac{1}{\sqrt{\eps_i}}-\frac{1}{\sqrt{\eps_j}})}=
\frac{\ln(\frac{\sqrt{\eps_j}}{\sqrt{\eps_i}})}{\sqrt{\alpha}(\frac{1}{\sqrt{\eps_i}}-\frac{1}{\sqrt{\eps_j}})}.\end{equation} Using the hypotheses it follows that
\[x_{i,j}< \frac{2\sqrt{\eps_i}}{\sqrt{\alpha}}\ln(\frac{\sqrt{\eps_j}}{\sqrt{\eps_i}})<
\frac{2\sqrt{\eps_i}}{\sqrt{\alpha}}\frac{\sqrt{\eps_j}}{\sqrt{\eps_i}}=\frac{2\sqrt{\eps_j}}{\sqrt{\alpha}}< \frac{1}{2}\] as required.\\
To prove (\ref{x4}), returning to (\ref{x5}) and writing $\eps_k = \exp(-p_k)$, for some $p_k
> 0$ and all $k$, it follows that
\[x_{i,j}=\frac{p_i -p_j}{\sqrt{\alpha}(\exp{p_i} -\exp{p_j})}.\] The
inequality $x_{i,j}< x_{i+1,j}$ is equivalent to
\[\frac{p_i -p_j}{\exp{p_i} -\exp{p_j}}<\frac{p_{i+1} -p_j}{\exp{p_{i+1}} -\exp{p_j}}, \]
which can be written in the form
\[(p_{i+1}-p_j)\exp(p_i-p_j)+(p_{i}-p_{i+1})-(p_{i}-p_j)\exp(p_{i+1}-p_j)>0. \]
With $a=p_i-p_j$ and $b=p_{i+1}-p_j$ it is not hard to see that $a>b>0$ and $a-b=p_i-p_{i+1}$. Moreover, the previous inequality is then equivalent to
\[\frac{\exp{a}-1}{a}>\frac{\exp{b}-1}{b}, \] which is true because $a>b$ and proves the first part of
(\ref{x4}). The second part is proved by a similar argument.\\ The analogous results for the $B^r_i$, $B^r_j$ and the points $1-x_{i,j}$ are proved by a similar argument.\\
\end{proof}
\section{The discrete problem}
A piecewise uniform mesh with $N$ mesh-intervals and mesh-points $\{x_i\}_{i=0}^N$ is now constructed by dividing the interval $[0,1]$ into $2n+1$ sub-intervals as
follows
\[[0,\sigma_1]\cup\dots\cup(\sigma_{n-1},\sigma_n]\cup(\sigma_n,1-\sigma_n]\cup(1-\sigma_n,1-\sigma_{n-1}]\cup\dots\cup(1-\sigma_1,1].\]
The $n$ transition parameters, which determine the points separating the uniform meshes, are defined by
\begin{equation}\label{sigma1}\sigma_{n}=\min\displaystyle\left\{\frac{1}{4},\sqrt{\frac{\eps_n}{\alpha}}\ln N\right\}\end{equation} and for $\;i=1,\;\dots \; ,n-1$
\begin{equation}\label{sigma2}\sigma_{i}=\min\displaystyle\left\{\frac{\sigma_{i+1}}{2},\sqrt{\frac{\eps_i}{\alpha}}\ln N\right\}.\end{equation} Clearly \[
0\;<\;\sigma_1\;<\;\dots\;<\;\sigma_n\;\le\;\frac{1}{4}, \qquad \frac{3}{4}\leq 1-\sigma_n < \; \dots \; 1-\sigma_1 <1.\] Then, on the sub-interval
$\;(\sigma_n,1-\sigma_n]\;$ a uniform mesh with $\;\frac{N}{2}\;$ mesh-intervals is placed, on each of the sub-intervals
$\;(\sigma_i,\sigma_{i+1}]\;\tx{and}\;(1-\sigma_{i+1},1-\sigma_i],\;\;i=1,\dots,n-1,\;$ a uniform mesh of $\;\frac{N}{2^{n-i+2}}\;$ mesh-intervals is placed and on
both of the sub-intervals $\;[0,\sigma_1]\;$ and $\;(1-\sigma_1,1]\;$ a uniform mesh of $\;\frac{N}{2^{n+1}}\;$ mesh-intervals is placed. In practice it is
convenient to take $N=2^{n}k$ where $k$
 is some positive power of 2. This construction leads to a class of $2^n$ piecewise uniform Shishkin
 meshes $M_{\vec{b}}$, where $\vec b$ denotes an $n$--vector with
$b_i=0$ if $\sigma_i=\frac{\sigma_{i+1}}{2}$ and $b_i=1$ otherwise. Note that $M_{\vec{b}}$ is a classical uniform mesh when $\vec b = \vec 0.$ It is not hard to
see also that
\begin{equation}\label{geom-1}
\sigma_i=2^{-(j-i+1)}\sigma_{j+1}\;\mathrm{when}\; b_i=\dots =b_j
=0.
\end{equation}
\begin{equation}\label{geom0}
B^l_i(\sigma_i)=B^r_i(1-\sigma_i)=N^{-1} \;\mathrm{when}\; b_i=1.
\end{equation}
 Writing $\delta_j=x_{j+1}-x _{j-1}$ note that, on any $M_{\vec b}$,
\begin{equation}\label{geom1}\delta_j \leq
CN^{-1}, \;\;\; 1 \leq j \leq N-1 \end{equation}
 and
\begin{equation}\label{geom2}\sigma_i \leq C \sqrt{\eps_i} \ln N, \; \;\; 1 \leq i \leq n. \end{equation}
Furthermore,
\begin{equation}\label{geom3} B_i^l(\sigma_i-(x_j-x_{j-1}))\leq CB_i^l(\sigma_i) \;\;\mathrm{if} \;\; x_j=\sigma_i,\;\; \mathrm{for\; some} \;\; i, j.
\end{equation}
To verify (\ref{geom3}) note that if $x_j=\sigma_i$ then $x_j-x_{j-1}=\frac{\sigma_i-\sigma_{i-1}}{(N/2^{n-i+2})}\leq N^{-1}\sigma_i 2^{n-i+2}$ and the result
follows from $B^l_i(\sigma_i-(x_j-x_{j-1}))\leq B^l_i(\sigma_i-\frac{N^{-1}\sigma_i}{2^{n-i+2}}) \leq \exp{(\frac{N^{-1}\ln N}{2^{n-i+2}})}B^l_i(\sigma_i)$ and
$\exp{(\frac{N^{-1}\ln N}{2^{n-i+2}})} \leq C.$\\
The discrete two-point boundary value problem is now defined on any mesh $M_{\vec b}$ by the finite difference method
\begin{equation}\label{discreteBVP}
-E\delta^2\vec{U} +A(x)\vec{U}=\vec{f}(x),  \qquad \vec{U}(0)=\vec{u}(0),\;\;\vec{U}(1)=\vec{u}(1).
\end{equation}
This is used to compute numerical approximations to the exact solution of (\ref{BVP}). Note that (\ref{discreteBVP}) can also be written in the operator form
\[\vec{L}^N \vec{U}\;=\;\vec{f}, \qquad \vec{U}(0)=\vec{u}(0),\;\;\vec{U}(1)=\vec{u}(1)\]
where \[\vec{L}^N\;=\;-E\delta^2+A(x)\] and $ \delta^2,\; D^+ \; \tx{and}
\; D^{-}$ are the difference operators
\[\delta^2\vec{U}(x_j)\;=\;\dfrac{D^+\vec{U}(x_j)-D^-\vec{U}(x_j)}{\oln{h}_j}\]
\[D^+\vec{U}(x_j)\;=\;\dfrac{\vec{U}(x_{j+1})-\vec{U}(x_j)}{h_{j+1}}\;\;\;\tx{and}\;\;\;D^-\vec{U}(x_j)\;=\;\dfrac{\vec{U}(x_j)-\vec{U}(x_{j-1})}{h_j}.\]
with $\;\oln{h}_j\;=\;\dfrac{h_j+h_{j+1}}{2},\;\;\;\;h_j\;=\;x_{j}-x_{j-1}.$\\
The following discrete results are analogous to those for the continuous case.
\begin{lemma}\label{dmax} Let $A(x)$ satisfy (\ref{a1}) and (\ref{a2}).
Then, for any mesh function $\vec{\Psi}$, the inequalities $\vec
{\Psi}(0)\;\ge\;\vec 0, \; \vec {\Psi}(1)\;\ge\;\vec 0
\;\rm{and}\;\vec{L}^N \vec{\Psi}(x_j)\;\ge\;\vec 0\;$ for
$1\;\le\;j\;\le\;N-1,\;$ imply that $\;\vec \Psi(x_j)\ge \vec 0\;$
for $0\;\le\;j\;\le\;N.\;$
\end{lemma}
\begin{proof} Let $i^*, j^*$ be such that
$\Psi_{i^*}(x_{j^{*}})=\min_{i,j}\Psi_i(x_j)$ and assume that the
lemma is false. Then $\Psi_{i^*}(x_{j^{*}})<0$ . From the hypotheses
we have $j^*\neq 0, \;N$ and $\Psi_{i^*}
(x_{j^*})-\Psi_{i^*}(x_{j^*-1})\leq 0, \; \Psi_{i^*}
(x_{j^*+1})-\Psi_{i^*}(x_{j^*})\geq 0,$ so
$\;\delta^2\Psi_{i*}(x_{j*})\;>\;0.\;$ It follows that
\[\ds\left(\vec{L}^N\vec{\Psi}(x_{j*})\right)_{i*}\;=
\;-\eps_{i*}\delta^2\Psi_{i*}(x_{j*})+\ds{\sum_{k=1}^n} a_{i*,\;k}(x_{j*})\Psi_{k}(x_{j*})\;<\;0,\] which is a contradiction, as required. \end{proof}

An immediate consequence of this is the following discrete stability result.
\begin{lemma}\label{dstab} Let $A(x)$ satisfy (\ref{a1}) and (\ref{a2}).
Then, for any mesh function $\vec \Psi $,
\[\parallel\vec \Psi(x_j)\parallel\;\le\;\max\left\{||\vec \Psi(0)||, ||\vec \Psi(1)||, \frac{1}{\alpha}||
\vec{L}^N\vec \Psi||\right\}, \; 0\leq j \leq N. \]
\end{lemma}
\begin{proof} Define the two functions
\[\vec{\Theta}^{\pm}(x_j)=\max\{||\vec{\Psi}(0)||,||\vec \Psi(1)||,\frac{1}{\alpha}||\vec{L^N}\vec{\Psi}||\}\vec{e}\pm
\vec{\Psi}(x_j)\]where $\vec{e}=(1,\;\dots \;,1)$ is the unit
vector. Using the properties of $A$ it is not hard to verify that
$\vec{\Theta}^{\pm}(0)\geq \vec{0}, \; \vec{\Theta}^{\pm}(1)\geq
\vec{0}$ and $\vec{L^N}\vec{\Theta}^{\pm}(x_j)\geq \vec{0}$. It
follows from Lemma \ref{dmax} that $\vec{\Theta}^{\pm}(x_j)\geq
\vec{0}$ for all $0\leq j \leq N$.
\end{proof}
\section{The local truncation error}
From Lemma \ref{dstab}, it is seen that in order to bound the error $||\vec{U}-\vec{u}||$ it suffices to bound $\vec{L}^N(\vec{U}-\vec{u})$. But this expression
satisfies
\[
\vec{L}^N(\vec{U}-\vec{u})=\vec{L}^N(\vec{U})-\vec{L}^N(\vec{u})= \vec{f}-\vec{L}^N(\vec{u})=\vec{L}(\vec{u})-\vec{L}^N(\vec{u})\]
\[=(\vec{L}-\vec{L}^N)\vec{u}
=-E(\delta^2-D^2)\vec{u}\] which is the local truncation of the second derivative. Then
\[E(\delta^2-D^2)\vec{u}
\;=\;E(\delta^2-D^2)\vec{v}+E(\delta^2-D^2)\vec{w}
\] and so, by the triangle inequality,
\begin{equation}\label{triangleinequality}
\parallel \vec L^N(\vec{U}-\vec{u})\parallel\;\leq\;\parallel
E(\delta^2-D^2)\vec{v}\parallel+\parallel E(\delta^2-D^2)\vec{w}\parallel.
\end{equation}  Thus, the smooth and singular components
of the local truncation error can be treated separately. In view of this it is noted that, for any smooth function $\psi$, the following two distinct estimates of
the local truncation error of its second derivative hold
\begin{equation}\label{lte1}
|(\delta^2-D^2)\psi(x_j)|\le2\max_{s\;\in\;I_j}|\psi^{\prime\prime}(s)|\;\;\;\;\;\qquad\qquad\;\;
\end{equation}
and
\begin{equation}\label{lte2}
|(\delta^2-D^2)\psi(x_j)|\;\le\;\dfrac{\delta_j}{3}\max_{s\;\in\;I_j}|\psi^{\prime\prime\prime}(s)|
\end{equation}
where $I_j=[x_{j-1},x_{j+1}]$.
\section{Error estimate}
The smooth component of the local truncation error is estimated in
the following lemma.
\begin{lemma} \label{smootherror} Let $A(x)$ satisfy (\ref{a1}) and (\ref{a2}). Then, for
each $i=1,\; \dots ,\; n$ and $j=1, \;\dots,\; N-1$, we have
\[ |\eps_{i}(\delta^2-D^2)v_i(x_j)| \leq\;C\sqrt{\eps_i}N^{-1}.\]
\end{lemma}
\begin{proof} Using (\ref{lte2}), Lemma \ref{smooth} and (\ref{geom1})
it follows that
\[|\eps_{i}(\delta^2-D^2)v_i(x_j)|\;\leq\;C\delta_j\ds\max_{s\;\in\;I_j}|\eps_{i} v_i^{\prime\prime\prime}(s)|\;\leq\;C\sqrt{\eps_i} \delta_j\;\leq\;C\sqrt{\eps_i}N^{-1}\]
as required.\eop\end{proof} For the singular component a similar
estimate is needed, but in the proof the different types of mesh
must be distinguished. The following preliminary lemmas are
required.
\begin{lemma}\label{est1}  Let $A(x)$ satisfy (\ref{a1}) and (\ref{a2}). Then,
for each $i=1,\; \dots ,\; n$ and $j=1, \;\dots,\; N$, on each mesh
$M_{\vec{b}}$, the following estimate holds  \[
|\eps_i(\delta^2-D^2)w^l_i(x_j)|\;\leq\;
\frac{C\delta_j}{\sqrt\eps_1}.
\]
\end{lemma}
\begin{proof}
From \eqref{lte2} and Lemma \ref{singular}, it follows that
\[
|\eps_i(\delta^2-D^2)w^l_i(x_j)|\leq
C\delta_j\;\ds\max_{s\;\in\;I_j}|\eps_{i}
w_i^{l,\prime\prime\prime}(s)|\]
\[\leq C\delta_j\;\ds\max_{s\;\in\;I_j}\ds\sum_{q\;=\;1}^n \dfrac{B^{l}_{q}(s)}{\sqrt\eps_q}\]
\[\leq \dfrac{C\delta_j}{\sqrt\eps_1}\] as required.\eop\end{proof}

In what follows  second degree polynomials of the form
\[p_{i;\theta}(x)=\sum_{k=0}^2
\frac{(x-x_{\theta})^k}{k!}w_{i}^{l,(k)}(x_{\theta})\] are used, where $\theta$ denotes a pair of integers separated by a comma.

\begin{lemma}\label{general} Let $A(x)$ satisfy (\ref{a1}) and (\ref{a2}). Then, for each
$i=1,\; \dots ,\; n$, $j=1, \;\dots,\; N$ and $k=1,\;\;\dots,\;\;
n-1$, on each mesh $M_{\vec{b}}$ with $b_k =1$, there exists a
decomposition \[ w^l_i=\sum_{m=1}^{k+1}w_{i,m}, \] for which  the
following estimates hold for each $m$,  $1 \le m \le k$,
\[|\eps_i
w_{i,m}^{\prime\prime}(x)| \leq CB^l_m(x), \;\;|\eps_i w_{i,m}^{\prime\prime\prime}(x)| \leq C\frac{B^l_{m}(x)}{\sqrt{\eps_m}}\] and
\[|\eps_i
w_{i,k+1}^{\prime\prime\prime}(x)| \leq C\sum_{q=k+1}^{n}\frac{B^l_{q}(x)}{\sqrt{\eps_q}}.\] Furthermore \[ |\eps_i(\delta^2-D^2)w^l_i(x_j)| \leq
C(B^l_{k}(x_{j-1})+\frac{\delta_j}{\sqrt{\eps_{k+1}}}).\] Analogous results hold for the  $w^r_i$ and their derivatives.
\end{lemma}
\begin{proof} Since $b_k =1$ it follows that $\sqrt{\eps_{k}}\leq \sqrt{\eps_{k+1}}/2$,  so
$x_{k,k+1} \in (0,\frac{1}{2})$ and the decomposition
\[w^l_i=\sum_{m=1}^{k+1}w_{i,m},\] exists, where the components of the
decomposition are defined by
\[w_{i,k+1}=\left\{ \begin{array}{ll} p_{i;k,k+1} & {\rm on}\;\;[0,x_{k,k+1})\\
 w^l_i & {\rm otherwise} \end{array}\right. \]
and for each $m$,  $k \ge m \ge 2$,
\[w_{i,m}=\left\{ \begin{array}{ll} p_{i;m-1,m} & \rm{on} \;\; [0,x_{m-1,m})\\
w^l_i-\sum_{q=m+1}^{k+1} w_{i,q} & {\rm otherwise}
\end{array}\right. \]
and
\[w_{i,1}=w^l_i-\sum_{q=2}^{k+1} w_{i,q}\;\; \rm{on} \;\; [0,1]. \]
From the above definitions it follows that, for each $m$, $1 \leq m \leq k$,
$w_{i,m}=0 \;\; \rm{on} \;\; [x_{m,m+1},1]$.\\
To establish the bounds on the third derivatives it is seen that:

for $x \in [x_{k,k+1},1]$, Lemma \ref{singular} and $x \geq
x_{k,k+1}$ imply that
\[|\eps_i w_{i,k+1}^{\prime\prime \prime}(x)| =|\eps_i w_{i}^{l,\prime\prime \prime}(x)| \leq
C\sum_{q=1}^n \frac{B^l_q(x)}{\sqrt{\eps_q}} \leq C\sum_{q=k+1}^n
\frac{B^l_q(x)}{\sqrt{\eps_q}};\]

for $x \in [0, x_{k,k+1}]$, Lemma \ref{singular} and $x \leq
x_{k,k+1}$ imply that
\[|\eps_i w_{i,k+1}^{\prime\prime \prime}(x)| =|\eps_i w_{i}^{l, \prime\prime
\prime}(x_{k,k+1})| \leq \sum_{q=1}^{n}
\frac{B^l_q(x_{k,k+1})}{\sqrt{\eps_q}} \leq \sum_{q=k+1}^{n}
\frac{B^l_q(x_{k,k+1})}{\sqrt{\eps_q}} \leq \sum_{q=k+1}^{n}
\frac{B^l_q(x)}{\sqrt{\eps_q}};\]

and for each $m=k, \;\; \dots \;\;,2$, it follows that\\

for $x \in [x_{m,m+1},1]$, $w_{i,m}^{\prime\prime \prime}=0;$

for $x \in [x_{m-1,m},x_{m,m+1}]$, Lemma \ref{singular} implies that
\[|\eps_i w_{i,m}^{\prime\prime \prime}(x)| \leq |\eps_i w_{i}^{l,\prime\prime
\prime}(x)|+\sum_{q=m+1}^{k+1}|\eps_i w_{i,q}^{\prime\prime
\prime}(x)| \leq C\sum_{q=1}^n \frac{B^l_q(x)}{\sqrt{\eps_q}} \leq
C\frac{B^l_m(x)}{\sqrt{\eps_m}};\]

for $x \in [0, x_{m-1,m}]$, Lemma \ref{singular} and $x \leq x_{m-1,m}$ imply that
\[|\eps_i w_{i,m}^{\prime\prime \prime}(x)| =|\eps_i w_{i}^{l,\prime\prime
\prime}(x_{m-1,m})| \leq C\sum_{q=1}^n \frac{B^l_q(x_{m-1,m})}{\sqrt{\eps_q}} \leq C\frac{B^l_m(x_{m-1,m})}{\sqrt{\eps_m}} \leq C\frac{B^l_m(x)}{\sqrt{\eps_m}};
\]

for $x \in [x_{1,2},1],\;\; w_{i,1}^{\prime\prime \prime}=0;$

for $x \in [0, x_{1,2}]$, Lemma \ref{singular} implies that \[|\eps_i w_{i,1}^{\prime\prime \prime}(x)| \leq |\eps_i w_{i}^{l,\prime
\prime\prime}(x)|+\sum_{q=2}^{k+1}|\eps_i w_{i,q}^{\prime\prime \prime}(x)|\leq C\sum_{q=1}^n \frac{B^l_q(x)}{\sqrt{\eps_q}} \leq C\frac{B^l_1(x)}{\sqrt{\eps_1}}.\]

For the bounds on the second derivatives note that, for each $m$, $1 \leq m \leq k $ :

for $x \in [x_{m,m+1},1],\;\; w_{i,m}^{\prime\prime}=0;$

for $x \in [0, x_{m,m+1}],\;\; \int_x^{x_{m,m+1}}\eps_i w_{i,m}^{\prime\prime \prime}(s)ds= \eps_i w_{i,m}^{\prime\prime}(x_{m,m+1})- \eps_i
w_{i,m}^{\prime\prime}(x)= -\eps_i w_{i,m}^{\prime\prime}(x)$ \\
and so
\[|\eps_i w_{i,m}^{\prime\prime}(x)| \leq \int_x^{x_{m,m+1}}|\eps_i
w_{i,m}^{\prime\prime \prime}(s)|ds \leq \frac{C}{\sqrt{\eps_m}}\int_{x}^{x_{m,m+1}} B^l_m(s)ds \leq CB^l_m(x).\] Finally, since
\[|\eps_i(\delta^{2}-D^2)w^l_i(x_j)| \leq \sum_{m=1}^{k}|\eps_i(\delta^{2}-D^2)w_{i,m}(x_j)|+ |\eps_i(\delta^{2}-D^2)w_{i,k+1}(x_j)|,\] using  (\ref{lte2}) on the
last term and (\ref{lte1}) on all other terms on the right hand side, it follows that
\[|\eps_i(\delta^{2}-D^2)w^l_i(x_j)| \leq C(\sum_{m=1}^{k}\max_{s \in
I_j}|\eps_iw_{i,m}^{\prime\prime}(s)| +\delta_j\max_{s \in
I_j}|\eps_iw_{i,k+1}^{\prime\prime \prime}(s)|).\]  The desired
result follows by applying the bounds on the derivatives in the
first part of this lemma.
\\The proof for the $w^r_i$ and their derivatives is similar. \end{proof}

%
\begin{lemma}\label{est3} Let $A(x)$ satisfy (\ref{a1}) and (\ref{a2}).
Then, for each $i=1,\; \dots ,\; n$ and $j=1, \;\dots,\; N$, on each mesh $M_{\vec{b}}$ the following estimate holds \[ |\eps_i(\delta^2-D^2)w^l_i(x_j)| \leq
CB^l_n(x_{j-1}).\] An analogous result holds for $w^r$.
\end{lemma}
\begin{proof}
From $\;\eqref{lte1}\;$ and Lemma \ref{singular}, for each
$\;i=1,\dots,n\;$ and $\;j=1,\dots,N,\;$ it follows that
\[|\eps_i(\delta^2-D^2)w^l_i(x_j)|\;\leq\;
C\;\ds\max_{s\in I_j}|\eps_i w_i^{l,\prime\prime}(s)|\;\]\[\le\; C\;\eps_i\ds\sum_{p=i}^{n}\dfrac{B^l_p(x_{j-1})}{\eps_p}\;\le\;CB^l_n(x_{j-1}).\eop\]
\end{proof}

Using the above preliminary lemmas on appropriate subintervals, the desired estimate of the singular components of the local truncation error are proved in the
following lemma.
\begin{lemma}\label{singularerror} Let $A(x)$ satisfy (\ref{a1}) and (\ref{a2}). Then,
for each $i=1,\; \dots ,\; n$ and $j=1, \;\dots,\; N$, the following
estimate holds
\[ |\eps_{i}(\delta^2-D^2)w_i(x_j)| \leq CN^{-1}\ln N. \]
\end{lemma}
\begin{proof} Since $\vec{w}=\vec{w}^l+\vec{w}^r$, it suffices to prove the result for $\vec{w}^l$ and $\vec{w}^r$ separately. Here it is proved for $\vec{w}^l$. A
similar proof holds for $\vec{w}^r$.\\
Stepping out from the origin each subinterval is treated separately.\\
First, consider $x \in (0,\sigma_1)$. Then, on each mesh $M_{\vec{b}}$,  $\delta_j \leq CN^{-1}\sigma_1$ and the result follows
 from (\ref{geom2}) and Lemma \ref{est1}.\\
Secondly, consider $x \in (\sigma_1,\sigma_2)$, then $\sigma_1 \leq x_{j-1}$ and $\delta_j \leq CN^{-1}\sigma_2$. The $2^{n+1}$ possible meshes are divided into 2
subclasses. On the meshes $M_{\vec b}$ with $b_1=0$ the result follows from (\ref{geom2}), (\ref{geom-1}) and Lemma \ref{est1}. On the meshes $M_{\vec b}$ with
$b_1=1$ the result follows from (\ref{geom2}), (\ref{geom0}) and Lemma \ref{general}. When $x=\sigma_1$, similar arguments apply for the 2 subclasses, except that
(\ref{geom3})is also needed for the second subclass.\\
Thirdly, in the general case $x \in (\sigma_m,\sigma_{m+1})$ for $2 \leq m \leq n-1$, it follows that $\sigma_m \leq x_{j-1}$ and $\delta_j \leq
CN^{-1}\sigma_{m+1}$. Then $M_{\vec b}$ is divided into 3 subclasses: $M_{\vec b}^0=\{M_{\vec b}: b_1= \dots =b_m =0\},\; M_{\vec b}^{r}=\{M_{\vec b}:  b_r=1, \;
b_{r+1}= \dots  =b_m =0 \; \mathrm{for \; some}\; 1 \leq r \leq m-1\}$ and $M_{\vec b}^m=\{M_{\vec b}: b_m=1\}.$ On $M_{\vec b}^0$ the result follows from
(\ref{geom2}), (\ref{geom-1}) and Lemma \ref{est1}; on $M_{\vec b}^r$ from (\ref{geom2}), (\ref{geom-1}), (\ref{geom0}) and Lemma \ref{general}; on $M_{\vec b}^m$
from (\ref{geom2}), (\ref{geom0}) and Lemma \ref{general}. When $x=\sigma_m$, similar arguments apply for the 3 subclasses, except that
(\ref{geom3})is also needed for the third subclass.\\
Finally, for $x \in (\sigma_n,1)$, $\sigma_n \leq x_{j-1}$ and $\delta_j \leq CN^{-1}$. Then $M_{\vec b}$ is divided into 3 subclasses: $M_{\vec b}^0=\{M_{\vec b}:
b_1= \dots  =b_n =0\},\; M_{\vec b}^{r}=\{M_{\vec b}:  b_r=1, \; b_{r+1}= \dots  =b_n =0 \; \mathrm{for \; some}\; 1 \leq r \leq n-1\}$ and $M_{\vec b}^n=\{M_{\vec
b}: b_n=1\}.$ On $M_{\vec b}^0$ the result follows from (\ref{geom2}), (\ref{geom-1}) and Lemma \ref{est1}; on $M_{\vec b}^r$ from (\ref{geom2}), (\ref{geom-1}),
(\ref{geom0}) and Lemma \ref{general}; on $M_{\vec b}^n$ from (\ref{geom0}) and Lemma \ref{est3}. When $x=\sigma_n$, similar arguments apply for the 3 subclasses,
except that (\ref{geom3})is also needed for the third subclass.\eop\end{proof}
 Let $\vec u$ denote the exact solution from
(\ref{BVP}) and $\vec U$ the discrete solution from (\ref{discreteBVP}). Then, the main result of this paper is the following parameter uniform error estimate
\begin{theorem} Let $A(x)$ satisfy (\ref{a1}) and (\ref{a2}). Then
 there exists a constant $C$ such that \[\parallel\vec{U}-\vec{u}\parallel \leq CN^{-1}\ln
 N,
\] for all $N > 1.$
\end{theorem}
\begin{proof}This follows immediately by applying Lemmas
\ref{smootherror} and \ref{singularerror} to (\ref{triangleinequality}) and using Lemma \ref{dstab}.\eop  \end{proof}
\section{An essentially second order method}
In this section it is shown that a simple modification to the Shishkin mesh constructed above leads to an essentially second order parameter-uniform numerical
method for (\ref{BVP}). The finite difference operator is the same as for the first order method; the Shishkin piecewise uniform mesh is modified by choosing a
different set of transition parameters. Instead of (\ref{sigma1}) and (\ref{sigma2}) the following parameters are used
\begin{equation}\label{tau1}\tau_{n}=\min\displaystyle\left\{\frac{1}{4},2\sqrt{\frac{\eps_n}{\alpha}}\ln N\right\}\end{equation} and for $\;i=1,\;\dots \; ,n-1$
\begin{equation}\label{tau2}\tau_{i}=\min\displaystyle\left\{\frac{\tau_{i+1}}{2},2\sqrt{\frac{\eps_i}{\alpha}}\ln N\right\}.\end{equation}
The proof that the resulting numerical method is essentially second order parameter-uniform is similar to the above and is based on an extension of the techniques
employed in \cite{MORSS}. It is assumed henceforth that the problem data satisfy additional smoothness conditions, as required.

It is first noted that, with the definitions (\ref{tau1}), (\ref{tau2}), (\ref{geom0}) is replaced by
\begin{equation}\label{geom20}
B^l_i(\tau_i)=B^r_i(1-\tau_i)=N^{-2} \;\mathrm{when}\; b_i=1.
\end{equation}
Also, for any smooth function $\psi$, in addition to (\ref{lte1}) and (\ref{lte2}), the following estimate of the local truncation error holds at any mesh point
$x_j$ of a locally uniform mesh
\begin{equation}\label{lte3}
|(\delta^2-D^2)\psi(x_j)|\;\le\;\dfrac{\delta_j^2}{12}\max_{s\;\in\;I_j}|\psi^{\prime\prime\prime\prime}(s)| \;\; \mathrm{if}\;\; x_j-x_{j-1}=x_{j+1}-x_{j}.
\end{equation}
From their construction,  it is clear that the above Shishkin meshes
are locally uniform everywhere, except at the points $x_j=\tau_k$
where $k \in I_{\vec{b}}$ and $I_{\vec{b}} = \{k: b_k=1\}.$\\ To
estimate the smooth component of the error, note that the estimate
in Lemma \ref{smootherror} can be modified to
\begin{eqnarray}\label{2ndordersmootherror}|L^N(V-v)_i(x_j)| 
\leq \; \left\{\begin{array}{l} C\sqrt{\eps_i}N^{-1}\;\;
\mathrm{if} \;\; x_j\in \{\tau_k,1-\tau_k\}, k \in I_{\vec{b}} \\
CN^{-2}\;\; \mathrm{otherwise}.\end{array}\right. \end{eqnarray} Now
introduce the mesh function $\vec{\Phi}$ where, for each $1 \leq i
\leq n$, \[\Phi_i(x_j) = CN^{-2}( \theta(x_j) +1),\] where
$\theta=\sum_{k \in I_{\vec{b}}} \theta_k $ and, for $k \in
I_{\vec{b}}$, $\theta_k$ is the
piecewise constant polynomial \begin{equation*} \theta_k(x)=  \left\{ \begin{array}{l}\;\; 0, \;\; x \in [0,\tau_k)\\\;\; 1, \;\; x \in [\tau_k, 1-\tau_k] \\
\;\; 0, \;\;  x \in (1-\tau_k,1] \end{array}\right
.\end{equation*}.\\Then \[0 \leq \Phi_i(x_j) \leq CN^{-2},\;\; 1
\leq i \leq n\] and
\begin{equation*}(L^N\Phi(x_j))_i =
CN^{-2}[-\eps_i\delta^2
\theta(x_j)+\sum_{j=1}^{n}a_{i,j}(x_j)(\theta(x_j)+1)].\end{equation*}
It follows that
\begin{equation}\label{lowerbound} (L^N\Phi(x_j))_i \ge \; \left\{\begin{array}{l} C^{\prime}(\eps_i +
N^{-2})\;\;
\mathrm{if} \;\; x_j\in \{\tau_k,1-\tau_k\}, k \in I_{\vec{b}} \\
C^{\prime}N^{-2}\;\; \mathrm{otherwise}.\end{array}\right.
\end{equation}
Considering the cases $\eps_i \ge N^{-1}$ and $\eps_i < N^{-1}$ separately,  choosing $C^{\prime}$ sufficiently small, comparing (\ref{lowerbound}) with
(\ref{2ndordersmootherror}) \;and applying the discrete maximum principle to the barrier functions
\[\vec{\Psi}^{\pm}=\vec{\Phi} \pm \vec{(V-v)}
\] gives the following estimate \begin{equation}\label{2ndordersmooth} ||\vec{(V-v)}|| \leq CN^{-2}. \end{equation}

To estimate the singular component of the error, note that the estimates in Lemmas \ref{est1} and \ref{general} are modified, respectively, to
\begin{eqnarray}
|\eps_i(\delta^2-D^2)w^l_i(x_j)|\;\leq\; \frac{C\delta_j^2}{\eps_1}\\
|\eps_i(\delta^2-D^2)w^l_i(x_j)| \leq C(B^l_{k}(x_{j-1})+\frac{\delta_j^2}{\eps_{k+1}})
\end{eqnarray}
Combining these with Lemma \ref{est3}, and repeating the same for the $w^r_i$ leads to the following estimate of the singular component of the local truncation
error
\begin{equation}
|\eps_{i}(\delta^2-D^2)w_i(x_j)| \leq C(N^{-1}\ln N)^2.
\end{equation}
Application of Lemma \ref{dstab} then gives
\begin{equation}\label{2ndordersingular}\;\; ||\vec{(W-w)}|| \leq C(N^{-1}\ln N)^2. \end{equation}
Combining (\ref{2ndordersmooth}) and (\ref{2ndordersingular}) leads at once to the required essentially second order parameter-uniform error estimate
\begin{equation}\;\; ||\vec{(U-u)}|| \leq C(N^{-1}\ln N)^2. \end{equation}

\end{document}